\documentclass[11pt]{amsart}
\usepackage{mathrsfs,latexsym,amsfonts,amssymb}
\usepackage{xcolor}
\usepackage{hyperref}
\usepackage{tikz-cd}
\hypersetup{colorlinks=true,
	linkcolor=cyan,
	urlcolor=red,
	citecolor=green,
}
\newtheorem{thm}{Theorem}[section]

\newtheorem{lemma}[thm]{Lemma}

\newtheorem{defn}[thm]{Definition}
\newtheorem{cor}[thm]{Corollary}

\newtheorem{prop}[thm]{Proposition}

\begin{document}	
	\title[Characterizations of quotient spaces for Lindel\"of strongly topological gyrogroups]
	{Characterizations of quotient spaces for Lindel\"of strongly topological gyrogroups}
	
	\author{Shumin Lai}
	\address{(Shumin Lai): School of mathematics and statistics, Minnan Normal University, Zhangzhou 363000, P. R. China}
	\email{shuminlai2026@163.com}
	
	\author{Fucai Lin*}
	\address{(Fucai Lin): 1. School of mathematics and statistics, Minnan Normal University, Zhangzhou 363000, P. R. China; 2. Fujian Key Laboratory of Granular Computing and Application, Minnan Normal University, Zhangzhou 363000, P. R. China}
	\email{linfucai2008@aliyun.com; linfucai@mnnu.edu.cn}

	\thanks{The authors are supported by Fujian Provincial Natural Science Foundation of China (No: 2024J02022) and the NSFC (No. 11571158).}
	\thanks{* Corresponding author}
	
	%insert keywords
	%insert subject class
	
	%\date{\today}

\maketitle	
\begin{abstract}
Let $\mathscr{L}$ be the class of Lindel\"of spaces such that $\mathscr{L}$ is closed under finite products. In this paper, we prove that if $G \in \mathscr{L}$ is a strongly topological gyrogroup, then $G$ is range-metrizable. Furthermore, we prove that if $H$ is a strong subgyrogroup of a strongly topological gyrogroup $G \in \mathscr{L}$, then every compact $G_\delta$-set in the quotient space $G/H$ is Dugundji.
Finally, for any strongly topological gyrogroup $G$ and any closed strong subgyrogroup $N$ of $G$, if $G \in \mathscr{L}$ and the quotient space $G/N$ is locally compact, then the inequality $w(G/N) \leq c$ is equivalent to the separability of $G/N$. Our results extend the classical results from topological groups to the class of strongly topological gyrogroups in the literature.
\end{abstract}

\section{Introduction}
Since the mid-20th century, the theory of topological groups has been a cornerstone of topological algebra.
Many topologists and algebraists have obtained numerous fundamental results.
Although the associativity of topological groups has been thoroughly investigated, recent developments in non-associative algebra motivate the exploration of weaker algebraic structures, among which gyrogroups attract extensive attention.
In 1988, A.A. Ungar gave the concept of gyrogroups to characterize the non-associative Einstein velocity addition law in special relativity, which arises from the relativistic Thomas precession effect. Indeed, A.A. Ungar said that this new mathematical formalism, is called ``gyrolanguage'', in which both Hyperbolic Geometry and Albert Einstein's Special Theory of Relativity find an aesthetically pleasing formulation under the same umbrella.
In 2008, the book ``Analytic Hyperbolic Geometry and Albert Einstein's Special Theory of Relativity'' was published, which systematically develops the algebraic properties of gyrogroups, derives the identity for gyroaddition, gyrotranslations and the gyroassociative law, and establishes a solid foundation for the topological research.
Thereafter, T. Suksumran in \cite{13} extended the classical group theory results including Cayley's, Lagrange's, and Isomorphism Theorems to the non-associative gyrocommutative setting of gyrogroups, formally developing the algebraic foundation of gyrogroups.

In \cite{2}, W. Atiponrat introduced the concept of topological gyrogroups.
In the following years,  researchers began to endow gyrogroups with the topological structures and establish the theory of topological gyrogroups.
J. Wattanapan et al. in \cite{15} verified that every locally compact Hausdorff topological gyrogroup $G$ can be embedded in a completely regular topological group $\Gamma$ as a twisted subset.
M. Bao and F. Lin defined the concept of strongly topological gyrogroup in \cite{6} and proved every $T_0$-strongly topological gyrogroup with a countable pseudocharacter admits a weaker metrizable topology in \cite{5}.
Based on the existing literature, we further extend the well-known and classical results of the topological groups theory to the broader framework of topological gyrogroups.

Let $\mathscr{L}$ be the class of Lindel\"of spaces such that $\mathscr{L}$ is closed under finite products. This paper is organized as follows.

In Section 2, we introduce
necessary notations and terminology which are used in
the paper.

In Section 3, we mainly prove that, for each open neighborhood $U$ of the neutral element $0$ in a strongly topological gyrogroup $G \in \mathscr{L}$ with a symmetric base $\mathscr{B}$ at $0$, there exists a continuous homomorphism $\pi$ of $G$ onto a metrizable strongly topological gyrogroup $H$ such that $\pi^{-1}(V) \subseteq U$, for some open neighborhood $V$ of the neutral element $0_H$ of $H$.

In Section 4, we prove that if $H$ is a strong subgyrogroup of a strongly topological gyrogroup $G \in \mathscr{L}$, then every compact $G_\delta$-set in the quotient space $G/H$ is Dugundji.
Moreover, for any strongly topological gyrogroup $G$ and any closed strong subgyrogroup $N$ of $G$, if $G \in \mathscr{L}$ and the quotient space $G/N$ is locally compact, then the inequality $w(G/N) \leq c$ holds if and only if $G/N$ is separable.

\section{Preliminaries}	
In this paper, all topological spaces are assumed to be Hausdorff.
Denote the sets of real numbers, rational numbers, positive integers and all non-negative integers by $\mathbb{R}$, $\mathbb{Q}$, $\mathbb{N}$ and $\omega$, respectively.
The weight $w(X)$ of a topological space $X$ is defined as the smallest cardinal number of the form $|\mathscr{B}|$, where $\mathscr{B}$ is a base of topology in $X$.
Denote by $\mathscr{L}$ the class of Lindel\"of spaces that is closed under finite products.
Readers may refer to \cite{1,7,16} for terminology and notations not explicitly given here.	

\begin{defn}[\cite{2}]\label{d-1.1}
	Let $(G, \oplus)$ be a groupoid. The system $(G, \oplus)$ is called a {\color{blue} gyrogroup}, if its binary operation satisfies the following conditions:
	\begin{itemize}
		\item[(G1)] There exists a unique identity element $0 \in G$ such that $0 \oplus a = a = a \oplus 0$ for all $a \in G$.		
		\item[(G2)] For each $x \in G$, there exists a unique inverse element $\ominus x \in G$ such that $\ominus x \oplus x = 0 = x \oplus (\ominus x)$.
		\item[(G3)] For all $x, y \in G$, there exists $\mathrm{gyr}[x,y] \in Aut(G, \oplus)$ with the property that $$x \oplus (y \oplus z) = (x \oplus y) \oplus \mathrm{gyr}[x,y](z)$$ for all $z \in G.$
		\item[(G4)] For any $x, y \in G$, $\mathrm{gyr}[x \oplus y, y] = \mathrm{gyr}[x,y]$.
	\end{itemize}
\end{defn}

\begin{defn}[\cite{2}]\label{d-1.2}
	Let $(G, \oplus)$ be a gyrogroup. A nonempty subset $H$ of $G$ is called a {\color{blue} subgyrogroup}, denoted by $H \leq G$, if the following statements hold:	
	
	\smallskip
	\indent\textup{(i)} The restriction $\oplus|_{H \times H}$ is a binary operation on $H$, i.e. $(H, \oplus|_{H \times H})$ is a groupoid.
		
	\smallskip
	\indent\textup{(ii)} For any $x, y \in H$, the restriction of $\mathrm{gyr}[x,y]$ to $H$, $\mathrm{gyr}[x,y]|_H:H \rightarrow \mathrm{gyr}[x,y](H)$, is a bijective homomorphism.	
	
	\smallskip
	\indent\textup{(iii)} $(H, \oplus|_{H \times H})$ is a gyrogroup.
		
	\smallskip
	\indent Furthermore, a subgyrogroup $H$ of $G$ is said to be an {\color{blue} $L$-subgyrogroup} \cite{12}, denoted by $H \leq_L G$, if $\mathrm{gyr}[a,h](H) = H$ for all $a \in G$ and $h \in H$.
\end{defn}

\begin{defn}[\cite{9}]\label{d-1.3}
A subgyrogroup $H$ of a topological gyrogroup $G$ is called {\color{blue} strong subgyrogroup} if for any $x, y \in G$, we have $\mathrm{gyr}[x,y](H) = H$.
\end{defn}

The subgyrogroup $H$ of $G$ being an $L$-subgyrogroup is a sufficient condition for $G$ to admit a disjoint partition into left cosets $\{a \oplus H: a \in G\}$.
Specifically, under this condition, the quotient space $G/H$ is well-defined, as shown in the following theorem:

\begin{thm}[\cite{12}]\label{d-1.5}
If $H$ is an $L$-subgyrogroup of a gyrogroup $G$, then the set $$\{a \oplus H: a \in G\}$$ forms a disjoint partition of $G$. 	
\end{thm}

\begin{defn}[\cite{12}]\label{d-1.4}
	The gyrogroup cooperation ``\,$\boxplus$'' is defined by the equation
	$$x \boxplus y = x \oplus \mathrm{gyr}[x, \ominus y](y), \quad x, y \in G.$$
\end{defn}

\begin{defn}[\cite{2}]\label{d-1.6}
	A triple $(G, \tau, \oplus)$ is called a {\color{blue} topological gyrogroup} if the following statements hold:
	
	\smallskip
	\indent\textup{(1)} $(G, \tau)$ is a topological space.
	
	\smallskip
	\indent\textup{(2)} $(G, \oplus )$ is a gyrogroup.
	
	\smallskip
	\indent\textup{(3)} The binary operation $\oplus: G \times G \rightarrow G$ is jointly continuous when $G \times G$ is endowed with the product topology, and the operation of taking the inverse $\ominus(\cdot): G \rightarrow G$, i.e. $x \rightarrow \ominus x$, is also continuous.
\end{defn}

We proceed to present the definition of strongly topological gyrogroup, which is crucial for the content of this paper.

\begin{defn}[\cite{6}]\label{d-1.7}
Let $G$ be a topological gyrogroup.
We say that $G$ is a {\color{blue} strongly topological gyrogroup} if there exists a neighborhood base $\mathscr{U}$ of 0 such that for all $U \in \mathscr{U}$, $gyr[x, y](U) = U$ for any $x, y \in G$.
For convenience, we say that $G$ is a {\it strongly topological gyrogroup with a neighborhood base $\mathscr{U}$ of 0}.
Moreover, we may assume that each element of $\mathscr{U}$ is symmetric.
\end{defn}	

\begin{defn}\cite{6}\label{d-1.8}
	Let $G$ be a gyrogroup, with the identity element $e$, and let $N$ be a real-valued function on $G$. Then $N$ is called a {\color{blue} prenorm} on $G$ if the following conditions hold for all $x, y\in G$:
	
	\smallskip
	\indent\textup{(PN1)} $N(e)=0$;
	
	\smallskip
	\indent\textup{(PN2)} $N(x \oplus y)\leq N(x)+N(y)$;
	
	\smallskip
	\indent\textup{(PN3)} $N(\ominus x)=N(x)$.
\end{defn}

\begin{defn}[\cite{1}]\label{t-77}
A compact space $X$ is called {\color{blue} Dugundji } if for every zero-dimensional compact space $Z$ and every continuous mapping $f:A \rightarrow X$, where $A$ is a closed subset of $Z$, there exists a continuous mapping $g: Z \rightarrow X$ extending $f$.
\end{defn}

The following proposition establishes two basic properties of Dugundji spaces.

\begin{prop}[\cite{1}]\label{t-7777}
	The class of Dugundji spaces has the following properties:
	\begin{itemize}
	\item[(a)] every compact metrizable space is Dugundji;
	\item[(b)] the product of an arbitrary family of Dugundji spaces is Dugundji.
	\end{itemize}
\end{prop}

\begin{thm}[\cite{1}]\label{t-777}
Every Dugundji space is dyadic.
\end{thm}

For later use, we recall the standard terminology for well-ordered inverse spectrum and its limit space.

\begin{defn}[\cite{1}]\label{d-8}
	Let $S = \{X_\alpha, p_\alpha^\beta: \alpha < \beta < \kappa\}$ be a family of spaces and mappings, where each $X_\alpha$ is a space and each $p_\alpha^\beta: X_\beta \rightarrow X_\alpha$ is a continuous mapping. We say that $S$ is a well-ordered {\color{blue} inverse spectrum} if all the mappings $p_\alpha^\beta$ satisfy the following commutativity condition:
	\begin{equation}\label{eq:p}
		p_\alpha^\beta \circ p_\beta^\gamma = p_\alpha^\gamma, \text{ whenever } \alpha < \beta < \gamma < \kappa.
		\tag{Com}
	\end{equation}
	All the mappings $p_\alpha^\beta$ are called connecting, while each mapping $p_\alpha^{\alpha+1}$ with $\alpha < \kappa$ is called bonding.
	It is a common practice to denote by $p_\alpha^\alpha$ the identity mapping of $X_\alpha$ onto itself, which permits to use the alternative notation $\{X_\alpha, p_\alpha^\beta:\alpha \leq \beta < \kappa\}$ for the spectrum $S$.
	
	To define the limit space of $S$, consider the product space $\Pi = \prod_{\alpha < \kappa} X_\alpha$ and denote by $\pi_\alpha$ the projection of $\Pi$ to the factor $X_\alpha$, where $\alpha < \kappa$.
	Let $X$ be the subspace of $\Pi$ defined by $$X = \{(x_\alpha)_{\alpha < \kappa} \in \Pi: p_\alpha^\beta(x_\beta) = x_\alpha \text{ whenever } \alpha < \beta < \kappa\}.$$
	Then $X$ is {\color{blue} the limit space} of the spectrum $S$ and every $x \in X$ is called a thread of $S$.
	Clearly, the restriction $p_\alpha = \pi_\alpha \!\restriction\! X$ is a continuous mapping of $X$ to $X_\alpha$, for each $\alpha < \kappa$; $p_\alpha$ is called a limit projection of $S$.
\end{defn}

The following extra conditions single out the class of continuous spectra, which plays an important role in the subsequent proofs.

\begin{defn}[\cite{1}]\label{d-88}
	Given an inverse spectrum $S = \{X_\alpha, p_\alpha^\beta: \alpha < \beta < \kappa\}$ and an ordinal $\gamma < \kappa$, let $S_\gamma = \{X_\alpha, p_\alpha^\beta: \alpha < \beta < \gamma\}$ be the initial part of $S$ of the length $\gamma$. The spectrum $S$ is called continuous if the following conditions are satisfied:
	\begin{itemize}
		\item[(C1)] $p_{\alpha}^{\alpha+1}(X_{\alpha+1}) = X_{\alpha}$, for each $\alpha < \kappa$;
		\item[(C2)] if $\beta > 0$ is a limit ordinal, then the diagonal product of the family $\{p_{\alpha}^{\beta} : \alpha < \beta\}$ is a homeomorphism of $X_{\beta}$ onto the limit space of the spectrum $S_{\beta}$.
	\end{itemize}
\end{defn}	

\begin{defn}[\cite{1}]\label{d-888}
	A continuous mapping $f:X \rightarrow Y$ has {\color{blue} metrizable kernel} if there exists a metrizable compact space $C$ and a topological embedding $i:X \rightarrow Y \times C$ such that $f = p \circ i$, where $p:Y \times C \rightarrow Y$ is the projection.
\end{defn}

\section{Continuous homomorphism into metrizable gyrogroup}
The work presented in this section generalizes Theorem 3.4.18 of \cite{1}. Indeed, in \cite{14}, we prove that Theorem 3.4.18 of \cite{1} holds in the class of $\sigma$-compact strongly topological gyrogroups, while now we shall prove that if $G \in \mathscr{L}$ is a strongly topological gyrogroup, then $G$ is range-metrizable.

To complete the proof of our main theorem (that is, Theorem~\ref{t-3}), we first give the following three technical lemmas.
\begin{lemma} \cite[Lemma 3.12]{6}\label{t-1}
Let $G$ be a strongly topological gyrogroup with the symmetric neighborhood base $\mathscr{U}$ at the identity element $0$, and let $\{U_n : n \in \omega\}$ and $\{V(m/2^n) : n, m \in \omega\}$ be two sequences of open neighborhoods satisfying the following conditions (1)-(5):
\begin{enumerate}
    \item[(1)] $U_n \in \mathscr{U}$ for each $n \in \omega$.
    \item[(2)] $U_{n+1} \oplus U_{n+1} \subseteq U_n$, for each $n \in \omega$.
    \item[(3)] $V(1) = U_0$;
    \item[(4)] For any $n \ge 1$, put
    $$ V(1/2^n) = U_n, \ V(2m/2^n) = V(m/2^{n-1}) $$
    for $m = 1, \dots, 2^{n-1}$, and
    $$ V((2m + 1)/2^n) = U_n \oplus V(m/2^{n-1}) = V(1/2^n) \oplus V(m/2^{n-1}) $$
    for each $m = 1, \dots, 2^{n-1} - 1$;
    \item[(5)] $V(m/2^n) = G$ when $m > 2^n$.
\end{enumerate}
Then there exists a prenorm $N$ on $G$ satisfies the following conditions:
\begin{enumerate}
    \item[(a)] for any fixed $x, y \in G$, we have $N(\mathrm{gyr}[x, y](z)) = N(z)$ for any $z \in G$;
    \item[(b)] for any $n \in \omega$,
    $$ \{x \in G : N(x) < 1/2^n\} \subseteq  U_n \subseteq \{x \in G : N(x) \le 2/2^n\}. $$
\end{enumerate}
\end{lemma}

Analogous to group theory, let $G$ be a gyrogroup, T. Suksumran defined a {\it normal subgyrogroup} $H \trianglelefteq G$ as the kernel of the gyrogroup homomorphism of $G$ in \cite{13}.
Any normal subgyrogroup $H$ of $G$ gives rise to the quotient gyrogroup $G/H$.

The following Lemma gives another characterization of normal subgyrogroups.

\begin{lemma}\cite[Theorem 31]{13}\label{t-1.1}
Let $H$ be a subgyrogroup of a gyrogroup $G$. Then $H \trianglelefteq G$ if and only if the operation on the coset space $G/H$ given by $$(a \oplus H) \oplus (b \oplus H) = (a \oplus b) \oplus H$$ is well defined.
\end{lemma}

\begin{lemma}\label{t-2}
Let $G \in \mathscr{L}$ be a strongly topological gyrogroup with a symmetric base $\mathscr{B}$ at the identity element $0$. Let $U$ be an arbitrary open neighborhood of $0$. Then there exist a countable family $\mathscr{C} \subseteq \mathscr{B}$ and a sequence $\{U_n\}_{n \in \omega} \subseteq \mathscr{C}$ such that the following conditions hold:
\begin{itemize}
\item[(1)] There exists $W_0 \in \mathscr{C}$ such that $U_0 \subseteq W_0 \subseteq U$;
\item[(2)] For each $n \in \omega$, we have $U_{n+1} \oplus U_{n+1} \subseteq U_n$;
\item[(3)] The sequence $\{U_n\}_{n \in \omega}$ is cofinal in the family $\mathscr{C}$; that is, for each $C \in \mathscr{C}$, there exists $n \in \omega$ such that $U_n \subseteq C$;
\item[(4)] The intersection of any finite subfamily of $\mathscr{C}$ belongs to $\mathscr{C}$;
\item[(5)] For each $C \in \mathscr{C}$ and all $a, b \in G$, there exists $D \in \mathscr{C}$ such that $$F_{a,b}(D \times D) \subseteq C,$$ where $$F_{a,b}(x,y) = \ominus(a \oplus b) \oplus  ((a \oplus x) \oplus (b \oplus y)).$$
\item[(6)] For each $C \in \mathscr{C}$ and all $a \in G$, there exists $D \in \mathscr{C}$ such that $I_a(D) \subseteq C$, where $$I_a(x) = a \oplus \ominus (a \oplus x).$$

\end{itemize}

Since the sequence $\{U_n\}_{n \in \omega}$ is cofinal in the family $\mathscr{C}$, the conditions (5) and (6) together are equivalent to the following condition:

\ \ \item[($\star$)] For each $C \in \mathscr{C}$ and all $a, b \in G$, there exists $k \in \omega$ such that $F_{a,b}(U_k \times U_k) \subseteq C$ and $I_a(U_k) \subseteq C$.
\end{lemma}

\begin{proof}
Define a mapping $\varphi: G \times G \times G \times G \rightarrow G$ by $\varphi(a, b, x, y) = \ominus (a \oplus b) \oplus ((a \oplus x) \oplus (b \oplus y))$ for each $a, b, x, y \in G$. Clearly, the mapping $\varphi$ is continuous and $\varphi(a, b, 0, 0) = 0$.
Put $$F_{a,b}(x, y) = \ominus (a \oplus b) \oplus ((a \oplus x) \oplus (b \oplus y)).$$ Then $F_{a,b}(0, 0) = 0$.
Fix any open neighborhood $W$ of $0$.

 Take any $(a, b) \in G \times G$. Then there exist open neighborhood $P_{a, b}$ of $a$, open neighborhood $Q_{a, b}$ of $b$ and open neighborhood $R_{a, b}\in \mathscr{B}$ of $0$ such that $$\varphi(P_{a,b} \times Q_{a, b} \times R_{a, b} \times R_{a, b}) \subseteq W.$$
That is, $\varphi(a', b', x, y) = F_{a', b'}(x, y) \in W$ for each $(a', b') \in P_{a, b}\times Q_{a, b}$ and every $(x, y) \in R_{a, b} \times R_{a, b}$.

Since $G \in \mathscr{L}$, we conclude that $G \times G$ is also Lindel\"of.
Then the open cover $\{P_{a,b} \times Q_{a, b}: (a, b) \in G \times G\}$ of $G \times G$ has a countable subcover, hence there exists a countable subset $\{a_{m, W}, b_{m, W}\in G: m\in\omega\}$ of $G$ such that the family $\{P_{a_{m, W}, b_{m, W}}\times Q_{a_{m, W}, b_{m, W}}: m \in \omega\}$ is a countable subfamily of $\{P_{a,b} \times Q_{a, b}: (a, b) \in G \times G\}$.
Then $$G \times G = \bigcup_{m \in \omega}(P_{a_{m, W}, b_{m, W}} \times Q_{a_{m, W}, b_{m, W}})$$ and $$\varphi(P_{a_{m, W}, b_{m, W}} \times Q_{a_{m, W}, b_{m, W}}\times R_{a_{m, W}, b_{m, W}}\times R_{a_{m, W}, b_{m, W}}) \subseteq W$$ for any $m\in\omega$.
Therefore, for each $a, b \in G$, there exists $m \in \omega$ such that $(a, b) \in P_{a_{m, W}, b_{m, W}} \times Q_{a_{m, W}, b_{m, W}}$ and
\begin{equation}\label{eq:Fab}
F_{a,b}(R_{a_{m, W}, b_{m, W}}\times R_{a_{m, W}, b_{m, W}}) \subseteq W.
\tag{1}
\end{equation}

Define a mapping $\phi: G \times G \rightarrow G$ by $\phi(c, x) = c \oplus \ominus (c \oplus x)$.
It is clear that $\phi$ is continuous. Then $\phi(c, 0) = 0$. For each $c \in G$, put $$I_c(x) = c \oplus \ominus (c \oplus x);$$ then $I_c(0) = 0$.
Take any $c \in G$. Then there exist open neighborhood $P_c$ of $c$ and open neighborhood $S_c \in \mathscr{B}$ of $0$ such that $$\phi(P_c \times S_c) \subseteq W.$$
Since $G$ is Lindel\"of, the open cover $\{P_c: c \in G\}$ of $G$ has a countable subcover $\{P_{c_{j, W}}: j \in \omega\}$ such that $G = \bigcup_{j \in \omega}P_{c_{j, W}}$ and $\phi(P_{c_{j, W}} \times S_{c_{j, W}}) \subseteq W$, where each $c_{j, W}\in G$.
Therefore, for each $c \in G$, there exists $j \in \omega$ such that $c \in P_{c_{j, W}}$ and \begin{equation}\label{eq:Ia}
    I_c(S_{c_{j, W}}) = \phi(c, S_{c_{j, W}}) \subseteq W.
\tag{2}
\end{equation}

Now we construct the countable family $\mathscr{C}$. For the open neighborhood $U$, it is obvious that there exists $W_0 \in \mathscr{B}$ such that $W_0 \subseteq U$.
By induction, now we construct the countable family $\mathscr{C}\subset\mathscr{L}$ satisfying the following conditions:
\\\indent\textup{(i)} $W_0 \in \mathscr{C}$;
\\\indent\textup{(ii)} For each $W \in \mathscr{C}$, there exists $V_W \in \mathscr{C}$ such that $V_W \oplus V_W \subseteq W$;
\\\indent\textup{(iii)} For each $W \in \mathscr{C}$, all the open neighborhoods $R_{a_{m, W}, b_{m, W}}$ and $S_{c_{j, W}}$ obtained from (\ref{eq:Fab}) and (\ref{eq:Ia}) are contained in the family $\mathscr{C}$;
\\\indent\textup{(iv)} The intersection of any finite subfamily of $\mathscr{C}$ belongs to $\mathscr{C}$.

Put $\mathscr{C}_0 = \{W_0\}$ and assume that by induction the family $\mathscr{C}_n$ has been defined for some $n\in\omega$.
For each $W \in \mathscr{C}_n$, there exists $V_W \in \mathscr{B}$ such that $V_W \oplus V_W \subseteq W$;
Now let
	\begin{align*}
		\mathscr{C}_{n+1}&=\{ V_W : W \in \mathscr{C}_n\} \cup\{ R_{a_{m, W}, b_{m, W}} : W \in \mathscr{C}_n, m\in\omega \}\\
		& \cup \{ S_{c_{j, W}}: W \in \mathscr{C}_n, j\in\omega\} \cup\{ \bigcap \lambda : \lambda \subseteq \mathscr{C}_n,\ |\lambda| < \omega\}.
	\end{align*}
Now put $\mathscr{C} = \bigcup_{n \in \omega}\mathscr{C}_n$. Clearly, the family $\mathscr{C}$ is countable and satisfies the conditions (i)-(v).

Enumerate the family as $\mathscr{C} = \{C_0, C_1, \cdots, C_n, \cdots \}$.
Without loss of generalization, we may assume $C_0 = W_0$.
We recursively construct a sequence $\{U_n\}_{n \in \omega}$ by choosing elements from the family $\mathscr{C}$ such that
\\\indent\textup{i$^{\prime}$)} $U_0 \subseteq W_0 \subseteq U$;
\\\indent\textup{ii$^{\prime}$)} $U_{n+1} \oplus U_{n+1} \subseteq U_n$ for each $n\in\omega$;
\\\indent\textup{iii$^{\prime}$)} The sequence $\{U_n\}_{n \in \omega}$ is cofinal in the family $\mathscr{C}$.

Indeed, let $U_0 = C_0$; then $U_0 \subseteq W_0 \subseteq U$.
Suppose that $U_n$ has been defined for some $n \in \omega$.
Put $D_n = U_n \cap C_{n+1}$; since the family $\mathscr{C}$ is closed under finite intersection, we have $D_n \in \mathscr{C}$. By (ii), there exists a symmetric neighborhood $U_{n+1} \in \mathscr{C}$ such that $U_{n+1} \oplus U_{n+1} \subseteq D_n$. Hence $U_{n+1} \oplus U_{n+1} \subseteq U_n$. It follows from the definition that $U_n \subseteq C_n$ for every $n \in \omega$, which implies that the sequence $\{U_n\}_{n \in \omega}$ is cofinal in $\mathscr{C}$.

Therefore, it is easy to verify that the family $\mathscr{C} \subseteq \mathscr{B}$ and the sequence $\{U_n\}_{n \in \omega}$ satisfy the conditions (1)-(6).
\end{proof}

Now we can prove our main theorem as follows.

\begin{thm}\label{t-3}
 Let $G \in \mathscr{L}$ be a strongly topological gyrogroup with a symmetric base $\mathscr{B}$ at the identity element $0$. Then for each open neighborhood $U$ of the neutral element $0$ in $G$, there exists a continuous homomorphism $\pi$ of $G$ onto a metrizable strongly topological gyrogroup $H$ such that $\pi^{-1}(V) \subseteq U$, for some open neighborhood $V$ of the neutral element $0_H$ of $H$.
\end{thm}

\begin{proof}
By Lemma~\ref{t-2}, there exist a countable family $\mathscr{C} \subseteq \mathscr{B}$ and a sequence $\{U_n\}_{n \in \omega} \subseteq \mathscr{C}$ of open neighborhoods at $0$ satisfying conditions (1)-(6) of Lemma~\ref{t-2}.

Put $Z = \bigcap_{n \in \omega}U_n$.
It is obvious that $Z$ is a subgyrogroup of $G$.
%{\blue For every $a, b \in Z$, then $a, b \in U_n$ for each $n \in \omega$. $a \oplus b \in U_{n+1} \oplus U_{n+1} \subseteq U_n$, for each $n \in \omega$. Therefore, $a \oplus b \in \bigcap_{n \in \omega} U_n = Z$. And $U_n = \ominus U_n$ for each $n \in \omega$, then $\ominus a \in \ominus U_n = U_n$ for each $n \in \omega$. Thus, $Z$ is a subgyrogroup of $G$.}
Since $U_n \in \mathscr{B}$ for each $n \in \omega$, we have $$\mathrm{gyr}[p,q](Z) = \mathrm{gyr}[p,q](\bigcap_{n \in \omega}U_n)\subset \bigcap_{n \in \omega} \mathrm{gyr}[p,q](U_n)= \bigcap_{n \in \omega}U_n = Z$$ for each $p, q \in G$, then $\mathrm{gyr}[p,q](Z)=Z$ for each $p, q \in G$ by \cite[Proposition 6]{12}. Therefore, $Z$ is a strong subgyrogroup of $G$.

Next we prove that $Z$ is a normal subgyrogroup of $G$. By Lemma~\ref{t-1.1}, it suffices to verify that the operation on the coset space $G/Z$ given by $$(a \oplus Z) \oplus (b \oplus Z) = (a \oplus b) \oplus Z$$ is well defined.
Suppose that $a_1 \oplus Z = a \oplus Z$ and $b_1 \oplus Z = b \oplus Z$,
that is, there exist $z_1, z_2 \in Z$ such that $a_1 = a \oplus z_1$ and $b_1 = b \oplus z_2$. Now, it remains to show that $(a_1 \oplus b_1) \oplus Z = (a \oplus b) \oplus Z$, which is equivalent to prove that $\ominus (a \oplus b) \oplus (a_1 \oplus b_1) \in U_n$ for each $n \in \omega$.

Take any $n \in \omega$. By the definition of $\mathscr{C}$, there exists $R \in \mathscr{C}$ such that $F_{a,b}(R \times R) \subseteq U_n$. Since the sequence $\{U_n\}_{n \in \omega}$ is cofinal in $\mathscr{C}$, there exists some $k \in \omega$ such that $U_k \subseteq R$. Since $z_1, z_2 \in Z = \bigcap_{n \in \omega}U_n$, it follows that $z_1, z_2 \in U_k$. Moreover, $z_1, z_2 \in R$. Therefore, $F_{a,b}(z_1, z_2) \in U_n$ for each $n \in \omega$. By the arbitrary choice of {\color{blue} $n\in\omega$}, we conclude that $F_{a,b}(z_1, z_2) \in Z$, that is, $$F_{a,b}(z_1, z_2) = \ominus (a \oplus b) \oplus ((a \oplus z_1) \oplus (b \oplus z_2)) = \ominus (a \oplus b) \oplus (a_1 \oplus b_1) \in Z.$$ Hence $(a_1 \oplus b_1) \oplus Z = (a \oplus b) \oplus Z$. Thus, $Z$ is a normal subgyrogroup of $G$. It follows that $H = G/Z$ is a gyrogroup by \cite[Theorem 27]{12}. Define the natural mapping $\pi: G \rightarrow H = G/Z$ by $\pi(x) = x \oplus Z$. By the definition of quotient operation, $\pi$ is a gyrogroup homomorphism from $G$ onto $H$.

Now, we shall define the topology on $H$.
The sequence $\{U_n\}_{n \in \omega}$ constructed above satisfy $\ominus U_n = U_n$ and $U_{n+1} \oplus U_{n+1} \subseteq U_n$ for each $n \in \omega$.
By Lemma~\ref{t-1}, there exists a continuous prenorm $N_G$ on $G$ such that $$N_G(\mathrm{gyr}[p,q](x)) = N_G(x)$$ for each $p, q, x \in G$
and
\begin{equation}\label{eq:N}
\{x \in G: N_G(x) < 1/2^n \} \subseteq U_n \subseteq \{x \in G: N_G(x) \leq 2/2^n \}
\tag{3}
\end{equation}
for each integer $n \geq 0$.

It follows from the definition of $N_G$, one can readily show that $Z = \{x \in G: N_G(x) = 0 \}$ and for each $x \in G$ and $z \in Z$, we have
\begin{equation}\label{eq:N_G}
N_G(x \oplus z) = N_G(x) = N_G(z \oplus x),
\tag{4}
\end{equation}
where the proof of the above equalities can be found in Theorem 3.1 of \cite{14}.

Define a function $\varrho: G \times G \rightarrow \mathbb{R}$ by
$$\varrho(x,y) = N_G(\ominus x \oplus y)$$
for any $x,y \in G$.
Clearly, $\varrho$ is continuous. It follows from Theorem 3.1 of \cite{14} that the function $\varrho$ is a left-invariant pseudometric on the gyrogroup $G$ and satisfies the additional condition:
\begin{equation}
	\tag{5}\label{eq:var}
	\text{For each } a, b \in G, a_1 \in a \oplus Z\text{ and } b_1 \in b \oplus Z\text{, the equality } \varrho(a_1, b_1) = \varrho(a, b) \text{ holds.}
\end{equation}

We also define a function $N_H$ on $H$ by the rule $$N_H(A) = N_G(a)$$ for each $A \in H$ and $a \in A$.
Then $N_H(\pi(x)) = N_G(x)$ for each $x \in G$. By (\ref{eq:N_G}), it is obvious that the function $N_H$ is well defined.
Since $\pi$ is a gyrogroup homomorphism of $G$ onto $H$ and $\pi(Z)$ is the identity element of the gyrogroup $H$,  one easily verifies that $N_H$ is a prenorm on $H$ satisfying the additional condition:
If $N_H(X) = 0$, then $X$ is the neutral element of $H$.
Moreover, for each $A, B, X \in H$, we claim that
\begin{equation}
	\tag{6}\label{eq:gyr}
	N_H(\mathrm{gyr}[A, B]X) = N_H(X)
\end{equation}
holds. Indeed, since $\pi$ is gyrogroup homomorphism, for each $A, B, X \in H$, there exist some $a \in A$, $b \in B$ and $x \in X$ such that $A = \pi(a)$, $B = \pi(b)$ and $\pi(x) = X$. Then $N_H(\mathrm{gyr}[A, B]X) = N_H(\mathrm{gyr}[\pi(a), \pi(b)]\pi(x))= N_H(\pi(\mathrm{gyr}[a,b]x)) = N_G(\mathrm{gyr}[a,b]x) = N_G(x) = N_H(X)$.

Next we define a function $d$ on $H$ by $$d(X,Y) = N_H(\ominus X \oplus Y)$$ for each $X, Y \in H$.
For each $X, Y \in H$, there exist some $x \in X$ and $y \in Y$ such that $X = \pi(x)$ and $Y = \pi(y)$. Hence $d(X, Y) = N_H(\ominus X \oplus Y) = N_H(\ominus \pi(x) \oplus \pi(y)) = N_H(\pi(\ominus x \oplus y)) = N_G(\ominus x \oplus y) = \varrho(x,y)$.
It is clear that $d$ is continuous.
By (\ref{eq:var}), the function $d$ is well defined.
One can easily show that $d$ is a metric on the gyrogroup $H$.

Endow $H$ with the topology $\mathscr{T}_H$ generated by the metric $d$.
For each $\varepsilon > 0$, put $$O(\varepsilon) = \{X \in H: N_H(X) < \varepsilon\}.$$
By the definition of $d$, we have $N_H(X) = N_H(\ominus 0_H \oplus X) = d(0_H, X)$; hence $O(\varepsilon) = \{X \in H: d(0_H, X) < \varepsilon\}$.
Now, for any $A \in H$ and $\varepsilon > 0$, we have $$B_d(A, \varepsilon) = \{Y \in H: d(A, Y) < \varepsilon\} = A \oplus O(\varepsilon).$$
%{\blue For each $Y \in B_d(A, \varepsilon)$, $d(A, Y) < \varepsilon$. It suffices to show that $\ominus A \oplus Y \in O(\varepsilon)$. $N_H(\ominus A \oplus Y) = d(A, Y) < \varepsilon$. Thus, $\ominus A \oplus Y \in O(\varepsilon)$. Converely, for each $Y \in O(\varepsilon)$, $N_H(Y) < \varepsilon$. It remains to prove that $A \oplus Y \in B_d(A, \varepsilon)$. $d(A, A \oplus Y) = N_H(\ominus A \oplus (A \oplus Y))  = N_H(Y) < \varepsilon$. Therefore, $A \oplus Y \in B_d(A, \varepsilon)$      }

It remains to verify that the multiplication and inverse operation on $H$ are continuous with respect to this metric topology $\mathscr{T}_H$.

First, we prove the multiplication is continuous under the topology $\mathscr{T}_H$.
Take any $A = \pi(a), B = \pi(b) \in H$ and  $\varepsilon > 0$, where $a, b \in G$. Next, we shall pick some $m\in\omega$ such that $B_d(A, 1/2^m)\oplus B_d(B, 1/2^m)\subset B_d(A \oplus B, \varepsilon)$. Clearly, there exists some $n \in \omega$ such that $2/2^n <\varepsilon$. By the same notations in Lemma~\ref{t-1.1}, we conclude that,
for any $U_n \in \mathscr{C}$, there exists $D\in \mathscr{C}$ such that $$F_{a,b}(D\times D) \subseteq U_n$$ for each $a, b \in G$.
Since the sequence $\{U_n\}_{n \in \omega}$ is cofinal in $\mathscr{C}$, there exists some $k \in \omega$ such that $U_k \subseteq D$, then $F_{a,b}(U_k \times U_k) \subseteq U_n$. From (\ref{eq:N}), it follows that we can find some $m \in \omega$ satisfying $$\{x \in G: N_G(x) < 1/2^m\} \subseteq U_k.$$
Therefore, $O(1/2^m) \subseteq \pi(U_k)$ holds.
%{\blue For each $Y \in O(\varepsilon)$, $N_H(Y) = N_H(\pi(y)) = N_G(y) < 1/2^m$. Then $y \in U_k$, $Y = \pi(y) \in \pi(U_k)$.}
Now pick any $X \in B_d(A, 1/2^m)$ and $Y \in B_d(B, 1/2^m)$. Since $$B_d(A, 1/2^m) = A \oplus O(1/2^m) \subseteq A \oplus \pi(U_k)$$ and $$B_d(B, 1/2^m) = B \oplus O(1/2^m) \subseteq B \oplus \pi(U_k),$$ there exist some $p, q \in U_k$ such that $X = A \oplus \pi(p) = \pi(a \oplus p)$ and $Y = B \oplus \pi(q) = \pi(b \oplus q)$,
then $X \oplus Y = \pi(a \oplus p) \oplus \pi(b \oplus q) = \pi((a \oplus p) \oplus (b \oplus q))$.

Moreover, we have $$F_{a,b}(p, q) \in F_{a,b}(U_k \times U_k) \subseteq U_n \subseteq \{x \in G: N_G(x) \leq 2/2^n\}.$$
Thus, $N_H(\pi(F_{a,b}(p,q))) =N_G(F_{a,b}(p,q)) \leq 2/2^n < \varepsilon$, then $\pi(F_{a,b}(p,q)) \in O(\varepsilon)$.
Since $F_{a,b}(p,q)= \ominus(a \oplus b) \oplus ((a \oplus p) \oplus (b \oplus q))$, it follows that $(a \oplus b) \oplus F_{a,b}(p,q) =(a \oplus p) \oplus (b \oplus q)$.
Hence, $$X \oplus Y = \pi((a \oplus p) \oplus (b \oplus q)) = \pi(a \oplus b) \oplus \pi(F_{a,b}(p,q)) \in (A \oplus B) \oplus O(\varepsilon) = B_d(A \oplus B, \varepsilon).$$ This proves that the multiplication is continuous at $(A, B)$. Since the arbitrary choices of $A$ and $B$, the multiplication on $H$ is continuous under the topology $\mathscr{T}_H$.

Next, we prove the continuity of the inverse mapping on $H$. Take any $A = \pi(a) \in H$ and $\varepsilon > 0$. It suffices to choose some $l \in \omega$ such that $\ominus B_d(A, 1/2^l) \subseteq B_d(\ominus A, \varepsilon)$. Indeed, it is obvious that there exists some $n \in \omega$ such that $2/2^n <\varepsilon$.
For $U_n \in \mathscr{C}$, there exists $L \in \mathscr{C}$ such that $$I_a(L) \subseteq U_n$$ for each $a \in G$.
Since $\{U_n\}_{n \in \omega}$ is cofinal in $\mathscr{C}$, there exists some $k \in \omega$ such that $U_k \subseteq L$, then $I_a(U_k) \subseteq U_n$.
By (\ref{eq:N}), we can find some $l \in \omega$ satisfying $$\{x \in G: N_G(x) < 1/2^l\} \subseteq U_k.$$
Therefore, $O(1/2^l) \subseteq \pi(U_k)$ holds.
For each $X \in B_d(A, 1/2^l)$, we have $B_d(A, 1/2^l) = A \oplus O(1/2^l) \subseteq A \oplus \pi(U_k)$. Hence there exists some $p \in U_k$ such that $X = A \oplus \pi(p) = \pi(a \oplus p)$, then $\ominus X = \pi(\ominus (a \oplus p))$.
Moreover, we have
$$I_a(p) \in I_a(U_k) \subseteq U_n \subseteq \{x \in G: N_G(x) \leq 2/2^n\}.$$
Thus, $N_H(\pi(I_a(p))) = N_G(I_a(p)) \leq 2/2^n < \varepsilon$, then $\pi(I_a(p)) \in O(\varepsilon)$.
Since $I_a(p) = a \oplus \ominus (a \oplus p)$, it follows that $\ominus a \oplus I_a(p) = \ominus (a \oplus p)$.
Hence, $$\ominus X = \pi(\ominus(a \oplus p)) = \pi(\ominus a \oplus I_a(p)) = \ominus A \oplus \pi(I_a(p)) \in \ominus A \oplus O(\varepsilon) = B_d(\ominus A, \varepsilon).$$
This proves that the inverse mapping is continuous on $H$.

Therefore, $H$ is a metrizable topological gyrogroup with respect to the topology $\mathscr{T}_H$.

We claim that $H$ is a strongly topological gyrogroup.
Since $O(\varepsilon) = B_d(0_H, \varepsilon)$,
the family $\{O(\varepsilon): \varepsilon > 0\}$ forms the local base at $0_H$ in $\mathscr{T}_H$.
Clearly, $O(\varepsilon) = \ominus O(\varepsilon)$ holds for any $\varepsilon > 0$ since $N_H(\ominus X) = N_H(X)$.
Moreover, by (\ref{eq:gyr}), we have $N_H(\mathrm{gyr}[A,B]X) = N_H(X) $ for each $A, B, X \in H$.
Hence, for every $\varepsilon > 0$, $${gyr}[A,B](O(\varepsilon)) = O(\varepsilon)$$
holds, since $\mathrm{gyr}[A,B]$ is bijective.
Therefore, $H$ is a metrizable strongly topological gyrogroup.

Finally, the equality $\pi^{-1}(O(\varepsilon)) = \pi^{-1}(\{X: N_H(X) <
\varepsilon\}) = \{x \in G: N_G(x) < \varepsilon\}$, where $\varepsilon > 0$, implies the homomorphism $\pi$ of $G$ onto $H$ is continuous at the identity element. Since $G$ and $H$ are topological gyrogroups, it follows that $\pi$ is continuous.
Notice also that if $x \in G$, $X = \pi(x)$, and $\varepsilon > 0$, then $N_G(x) < \varepsilon$ is equivalent to $N_H(X) < \varepsilon$. Therefore, $\pi^{-1}(\{X: N_H(X) <
\varepsilon\}) = \{x \in G: N_G(x) < \varepsilon\}$. In particular, put $V = O(1)$. Then we have $\pi^{-1}(O(1)) = \{x \in G: N_G(x) < 1\} \subseteq U_0 \subseteq U$.

As the continuous image of a Lindel\"of space $G$, the space $H = \pi(G)$ is also Lindel\"of.
Furthermore, since $H$ is metrizable, it follows that $H$ is second countable.
\end{proof}

\begin{defn}\cite[Definition 3.2]{14}\label{d-15}
Let $\mathscr{P}$ be a class of topological gyrogroups, and let $G$ be any topological gyrogroup. Let us say that $G$ is {\color{blue} range-$\mathscr{P}$} if for every open neighborhood $U$ of the identity element $0_G$ of $G$, there exists a continuous homomorphism $\pi$ of $G$ to a gyrogroup $H \in \mathscr{P}$ such that $\pi^{-1}(V) \subset U$, for some open neighborhood $V$ of the identity element $0_H$ of $H$.
\end{defn}

From the definition above, one can deduce that if $G \in \mathscr{L}$ is a strongly topological gyrogroup, then $G$ is range-metrizable.

If $G$ is a Lindel\"of-$\Sigma$ strongly topological gyrogroup, it follows from \cite[Proposition 5.3.9]{1} that $G \times G$ is also Lindel\"of-$\Sigma$, hence $G\in\mathscr{L}$. According to Theorem~\ref{t-3}, the following corollary is obvious.

\begin{cor}\label{t-4}
Let $G$ be a Lindel\"of-$\Sigma$ strongly topological gyrogroup with a symmetric base $\mathscr{B}$ at the identity element $0$. Then $G$ is range-metrizable.
\end{cor}

\section{Some properties of the quotient space and inverse spectrum}
In this section, we mainly prove that if $H$ is a strong subgyrogroup of a strongly topological gyrogroup $G \in \mathscr{L}$, then every compact $G_\delta$-set in the quotient space $G/H$ is Dugundji.
Furthermore, let $G \in \mathscr{L}$ be a strongly topological gyrogroup and $N$ be a closed strong subgyrogroup of $G$. If the quotient space $G/N$ is locally compact, then the inequality $w(G/N) \leq c$ is equivalent to the separability of $G/N$. First, we give some concepts and technical lemmas or facts.

Recall that a Tychonoff space $X$ admitting a continuous bijection $f: X \rightarrow M$ onto a metrizable space $M$ is said to be {\it submetrizable}, and that $f$ is said to be a {\it condensation} of $X$ onto $M$.

We fix a topological gyrogroup $G$ and its strong subgyrogroup $H$. Let $\pi:G \rightarrow G/H$ be the quotient mapping of $G$ onto the space $G/H$ of the left cosets of $H$ in $G$.
Denote by $\mathscr{K}$ be the family of all $L$-subgyrogroup $K$ of $G$ such that $H \subseteq K$ and the quotient space $G/K$ is submetrizable.
If $K, L \in \mathscr{K}$ and $K \subseteq L$, let $\pi_K: G \rightarrow G/K$, $\pi_L: G \rightarrow G/L$ and $\pi_L^K: G/K \rightarrow G/L$ be the natural mappings, where $\pi_L^K(x \oplus K) = x \oplus L$ for each $x \in G$. Since $\pi_L = \pi_L^K \circ \pi_K$ and both mappings $\pi_K$ and $\pi_L$ are open, so is $\pi_L^K$.
%{\color{red} Suppose that $U$ is an open set of $G/K$, our aim is to verify that $\pi_L^K(U)$ is open in $G/L$. $\pi_L^K(U) = \pi_L^K(\pi_K(\pi^{-1}_K(U))) = \pi_L(\pi^{-1}_K(U))$ is open in $G/L$ since $\pi^{-1}_K(U)$ is open in $G$ and $\pi_L$ is a  open mapping. And $\pi_L^K$ is a continuous mapping. Suppose that $V$ is open in $G/L$, our aim is to show that $(\pi_L^K)^{-1}(V)$ in open in $G/K$. $(\pi_L^K)^{-1}(V) = (\pi_L^K)^{-1}(\pi_L(\pi^{-1}_L(V))) = \pi_K(\pi^{-1}_L(V))$ is open in $G/K$ since $\pi_L$ is continuous and $\pi_K$ is open.}

\begin{lemma}\label{t-6}
	If $\gamma$ is a countable subfamily of $\mathscr{K}$, then $\bigcap\gamma \in \mathscr{K}$.
\end{lemma}
\begin{proof}
	Put $N = \bigcap\gamma$. Clearly, $H \subseteq N$.
	Take any $x \in G$ and $y \in N$. Then
	$$\mathrm{gyr}[x, y](N) = \mathrm{gyr}[x, y](\bigcap_{K \in \gamma}K) \subseteq \bigcap_{K \in \gamma}\mathrm{gyr}[x, y](K) = \bigcap_{K \in \gamma}K = \bigcap \gamma = N.$$ Hence $\mathrm{gyr}[x, y](N)=N$ for each $x \in G$ and $y\in N$ by \cite[Proposition 6]{12}.
	Therefore, $N$ is a $L$-subgyrogroup of $G$.
	It suffices to verify that the quotient space $G/N$ is submetrizable.
	For every $K \in \gamma$, the submetrizable space $G/K$ is an image of $G/N$ under the continuous mapping $\pi_K^N$, so the diagonal product $\varphi$ of the family $\{\pi_K^N: K \in \gamma\}$ is a continuous injective mapping of $G/N$ to the product space $\Pi = \prod_{K \in \gamma}G/K$.
	%{\color{red} $\varphi$ is continuous, since every $\pi_K^N$ is continuous. Why $\varphi$ is injective? Suppose $\varphi(x \oplus N) = \varphi(y \oplus N)$, then $\pi_K^N(x \oplus N) = \pi_K^N(y \oplus N)$ for each $K \in \gamma$. Then $x \oplus K = y \oplus K$ for each $K \in \gamma$, $\ominus y \oplus x \in K$ for each $K \in \gamma$, therefore, $\ominus y \oplus x \in N = \bigcap\gamma$, that is, $x \oplus N = y \oplus N$. Thus, $\varphi$ is injective. }
	The space $\Pi$ is submetrizable, as a countable product of submetrizable spaces.
	Since $\varphi$ is a continuous injective mapping and any subspace of a submetrizable space is submetrizable, the space $G/N$ is also submetrizable, that is, $N \in \mathscr{K}$.%{\red  单向的连续双射可以传递submetrizable. }
\end{proof}

\begin{prop}\cite[Proposition 1.4.8]{7}\label{t-66}
	Suppose we are given a set $X$, a family $\{Y_s\}_{s \in S}$ of topological spaces and a family of mappings $\{f_s\}_{s \in S}$, where each $f_s$ is a mapping of $X$ to $Y_s$.
	In the class of all topologies on $X$ that make all the $f_s$'s continuous there exists a coarsest topology; this is the topology $\mathscr{O}$ generated by the base consisting of all sets of the form $\bigcap_{i = 1}^k f_{s_i}^{-1}(V_i)$, where $s_1$, $s_2$, $\dots$, $s_k \in S$ and $V_i$ is an open subset of $Y_{s_i}$ for $i = 1,2, \dots, k$.
	
	The topology $\mathscr{O}$ is called the topology generated by the family of mappings $\{f_s\}_{s \in S}$.
\end{prop}

The following Lemma is important in the proof of our main theorem in this section.

\begin{lemma}\label{t-7}
	If $G \in \mathscr{L}$ is a strongly topological gyrogroup with the symmetric neighborhood base $\mathscr{N}_0$ at the identity element $0$ of $G$, then the family $\{\pi_K^H: K \in \mathscr{K}\}$ generates the topology of the quotient space $G/H$.
\end{lemma}

\begin{proof}
	Since $H$ is a strong subgyrogroup of $G$, it follows from \cite[Theorem 3.13]{9} that $G/H$ is homogeneous.
	Take an arbitrary fixed open neighborhood $U$ of $0$ in $G$.
	Since the quotient $\pi: G \rightarrow G/H$ is open, the family $\{\pi(O): O \in \mathscr{N}_0\}$ is a neighborhood base at $\pi(0)$ in $G/H$.
	Thus, it suffices to find a neighborhood $V$ of $0$ in $G$ and $K \in \mathscr{K}$ such that $V \oplus K \subseteq U \oplus H$. This will imply that $(\pi_K)^{-1} \pi_K(V) \subseteq \pi^{-1}\pi(U)$
	and then $\pi(\pi_K^{-1}(\pi_K(V)))\subseteq \pi(U)$.
	Indeed, since $\pi_K = \pi_K^H \circ \pi$,
	it follows that $(\pi_K^H)^{-1}(\pi_K(V)) \subseteq \pi(U)$.

	Since $G \in \mathscr{L}$ is a strongly topological gyrogroup, by Theorem~\ref{t-3}, one can find a continuous homomorphism $\varphi: G \rightarrow M$ onto a metrizable strongly topological gyrogroup $M$, and an open neighborhood $W$ of the identity element in $M$ such that $\varphi^{-1}(W) \subseteq U$.
	Furthermore, $H$ is a strong subgyrogroup, then, for each $x, y \in G$, we have $\mathrm{gyr}[x, y](H) = H$.
	For each $a, b \in M$, there exist $x, y \in G$ such that $\varphi(x) = a$ and $\varphi(y) = b$, then by \cite[Proposition 23]{12} we have
	$$\mathrm{gyr}[a, b](\varphi(H)) = \mathrm{gyr}[\varphi(x), \varphi(y)](\varphi(H)\bigr) = \varphi(\mathrm{gyr}[x, y](H)) = \varphi(H).$$
	Thus $\varphi(H)$ is a strong subgyrogroup of $M$.
	Denote by $N$ the closure of $\varphi(H)$ in $M$.
	Since $\varphi$ is a homomorphism and $\varphi(H)$ is a subgyrogroup of $M$, then by \cite[Proposition 7]{2}, $N$ is also a subgyrogroup of $M$.
	It follows from \cite[Lemma 4]{2} that $\mathrm{gyr}[a, b]$ is a homeomorphism for any $a, b \in G/H$. Hence
	$$\mathrm{gyr}[a, b](N)
	= \mathrm{gyr}[a, b](\overline{\varphi(H)})
	= \overline{\mathrm{gyr}[a, b](\varphi(H))}
	= \overline{\varphi(H)}
	= N.$$
	Therefore, $N$ is also a strong subgyrogroup of M.
	Moreover, $K = \varphi^{-1}(N)$ is a closed strong subgyrogroup of $G$, and $H \subseteq K$.
	For any $x, y \in G$, $$\varphi(\mathrm{gyr}[x,y](K)) = \mathrm{gyr}[\varphi(x), \varphi(y)](\varphi(K)) = \mathrm{gyr}[\varphi(x), \varphi(y)](N) = N.$$ Then $\mathrm{gyr}[x,y](K) \subseteq \varphi^{-1}(N) = K$.
	By \cite[Proposition 6]{12},  we have $\mathrm{gyr}[x,y](K) = K$. Thus, $K$ is a closed strong subgyrogroup.
	Finally, $\varphi(H) \subseteq N = \overline{\varphi(H)}$; therefore, $H \subseteq \varphi^{-1}\varphi(H) \subseteq \varphi^{-1}(N) = K$.
	
	%{\color{blue} For each $a, b \in \varphi(H)$, there exists $x, y \in H$ such that $\varphi(x) = a$, $\varphi(y) = b$. Since $H$ is a subgyrogroup, $\ominus x, x \oplus y \in H$, $\varphi(\ominus x) = \ominus \varphi(x) = \ominus a \in \varphi(H)$, $\varphi(x \oplus y) = \varphi(x) \oplus \varphi(y) = a \oplus b \in H$. $M$ is a topological gyrogroup, $\varphi(H)$ is a subgyrogroup of $M$, then $N = \overline{\varphi(H)}$ is also a gyrogroup of $M$. $K = \varphi^{-1}(N)$ is a closed subgyrogroup of $G$. For each $a, b \in K$, there exists $x, y \in N$ such that $\varphi(a) = x$, $\varphi(b) = y$. Since $\varphi(\ominus a) = \ominus\varphi(a) = \ominus x \in N$, $\ominus a \in \varphi^{-1}(N) = K$. And $\varphi(a \oplus b) = \varphi(a) \oplus \varphi(b) = x \oplus y \in N$, then $a \oplus b \in \varphi^{-1}N = K$. Finally, $\varphi(H) \subseteq N = \overline{\varphi(H)}$, therefore, $H \subseteq \varphi^{-1}\varphi(H) \subseteq \varphi^{-1}(N) = K$. }
	
	Further, let $i: G/K \rightarrow M/N$ be a mapping defined by the rule $i(x \oplus K) = \varphi(x) \oplus N$, for each $x \in G$. We claim that the mapping $i$ is defined correctly. Indeed,
	if $x \oplus K = y \oplus K$, then $\ominus y \oplus x \in K$. $\varphi(\ominus y \oplus x) \in \varphi(K) = \varphi(\varphi^{-1}(N)) \subseteq N$, then $\ominus \varphi(y) \oplus \varphi(x) \in N$, that is, $\varphi(x) \oplus N = \varphi(y) \oplus N$.

	Moreover, since the mapping $\pi_K$ is open and $i \circ \pi_K = p_N \circ \varphi$, where $p_N: M \rightarrow M/N$ is the canonical quotient mapping, the mapping $i$ is continuous. Indeed, let $U$ be an arbitrary open set in $M/N$, it suffices to show that $i^{-1}(U)$ is open in $G/K$. Since $i \circ \pi_K = p_N \circ \varphi$, we observe that $\pi_K^{-1}\left( i^{-1}(U) \right) = (i \circ \pi_K)^{-1}(U) = (p_N \circ \varphi)^{-1}(U)$. Since $U$ is open and $p_N \circ \varphi$ is continuous, we conclude that $\pi_K^{-1}(i^{-1}(U))$ is open in $G$, which implies that $i^{-1}(U)$ is open in $G/K$ by the definition of the quotient topology. Thus, $i$ is continuous.
	
	We claim that $i$ is a bijection of $G/K$ onto $M/N$. Indeed, suppose that $x, y \in G$ and $i(x \oplus K) = i(y \oplus K)$. Then $\varphi(x) \oplus N = \varphi(y) \oplus N$, and consequently, $\varphi(\ominus x \oplus y) \in N$ because $N$ is a strong subgyrogroup. Hence, $\ominus x \oplus y \in \varphi^{-1}(N) = K$ and $\pi_K(x) = \pi_K(y)$, that is, $x \oplus K = y \oplus K$.

	Since $G$ is Lindel\"of, it follows that $M = \varphi(G)$ as its continuous image is also Lindel\"of. Thus, $M$ is second-countable. Since $N$ is a strong subgyrogroup of $M$, the quotient mapping $p_N$ is both open and surjective. Consequently, $M/N$ is also second-countable by applying \cite[Theorem 1.4.16]{7}. Further, it follows from \cite[Theorem 3.7]{10} that the quotient space $M/N$ is regular. By \cite[Theorem 4.2.9]{7}, we conclude that $M/N$ is metrizable.
	Thus, $i$ is a condensation of $G/K$ onto the quotient space $M/N$ which is metrizable.
	Therefore, $K \in \mathscr{K}$.
	
	Finally, $V = \varphi^{-1}(W)$ is an open neighborhood of $0$ in $G$ and $V \subseteq U$. Since $W$ is open in $M$ and $\varphi(H)$ is dense in $N$, we claim
\begin{align*}
	    		V \oplus K&= \varphi^{-1}(W) \oplus \varphi^{-1}(N)\\
                &= \varphi^{-1}(W \oplus N)\\
                &= \varphi^{-1}(W \oplus \varphi(H))\\
	    		&=\varphi^{-1}(W) \oplus H\\
                &\subseteq U \oplus H.
	    	\end{align*}

Indeed, since $\varphi$ is a surjective homomorphism, we have $\varphi^{-1}(W) \oplus \varphi^{-1}(N) = \varphi^{-1}(W \oplus N)$.
Moreover, it is obvious that  $\varphi^{-1}(W \oplus \varphi(H)) \subseteq \varphi^{-1}(W \oplus N)$. Next we clarify that $W \oplus N \subseteq W \oplus \varphi(H)$. Take any $x \in W \oplus N$. Let $x = w \oplus c$, where $w \in W$ and $c \in N$. The set $U = (\ominus W \oplus x) \cap N$ is an open neighborhood of $c$ in $N$ since $c= \ominus w \oplus x \in \ominus W \oplus x$.
By the density of $\varphi(H)$ in $N$, we conclude that $((\ominus W \oplus x) \cap N) \cap \varphi(H) \neq \emptyset$, so there exists $h' = \ominus u \oplus x \in (\ominus W \oplus x) \cap \varphi(H)$ for some $u \in W$, which implies $x = u \oplus h' \in W \oplus \varphi(H)$. Thus, $W \oplus N \subseteq W \oplus \varphi(H)$. Then $\varphi^{-1}(W \oplus N) \subseteq \varphi^{-1}(W \oplus \varphi(H))$.
Clearly, $\varphi^{-1}(W) \oplus H \subseteq \varphi^{-1}(W \oplus \varphi(H)) = \varphi^{-1}(W) \oplus \varphi^{-1}(\varphi(H))$.
	Conversely, take any $g \in \varphi^{-1}(W \oplus \varphi(H))$; then $\varphi(g) = w_{1} \oplus \varphi(h) \in W \oplus \varphi(H)$ for some $w_{1} \in W$ and $h \in H$. Since
\begin{align*}
	    		\varphi(g) \boxminus \varphi(h)&= \varphi(g) \oplus \mathrm{gyr}[\varphi(g), \varphi(h)](\ominus \varphi(h))\\
                &= \varphi(g \oplus \mathrm{gyr}[g,h](\ominus h))\\
                &= \varphi(g \boxminus h)\\
	    		&= w_{1} \in W,
	    	\end{align*}
it follows that  $g \boxminus h \in \varphi^{-1}(W)$. Hence $g \in \varphi^{-1}(W) \oplus h \subseteq \varphi^{-1}(W) \oplus H$.
	Thus, $\varphi^{-1}(W \oplus \varphi(H)) \subseteq \varphi^{-1}(W) \oplus H$.
	Then $\varphi^{-1}(W \oplus \varphi(H)) = \varphi^{-1}(W) \oplus H$.
	
Therefore, $V \oplus K \subseteq U \oplus H$. As we mentioned at the beginning of the proof, this means that the family $\{\pi_K^H: K \in \mathscr{K}\}$ generates the topology at the point $\pi(0)$ of the quotient space $G/H$. The conclusion of the lemma follows, since the natural action of $G$ on $G/H$ is continuous and transitive.
	
	%{\color{blue} $V \oplus K = \varphi^{-1}(W) \oplus \varphi^{-1}(N) =  \varphi^{-1}(W \oplus N)$, first we prove:
		%$\varphi^{-1}(W) \oplus \varphi^{-1}(N) \subseteq \varphi^{-1}(W \oplus N)$
		%For each $x \in \varphi^{-1}(W)$, and $y \in \varphi^{-1}(N)$, then $\varphi(x \oplus y) = \varphi(x) \oplus \varphi(y) \in W \oplus N$, then $x \oplus y \in \varphi^{-1}(W \oplus N)$.
		%Next we prove $\varphi^{-1}(W \oplus N) \subseteq \varphi^{-1}(W) \oplus \varphi^{-1}(N)$.
		%For each $z \in \varphi^{-1}(W \oplus N)$, then $\varphi(z) = w \oplus n$, where $w \in W$, $n \in N$. Since $\varphi$ is surjective, there exists $x \in G$ such that $\varphi(x) = w$, that is, $x \in \varphi^{-1}(W)$. For $\ominus x \oplus z$, $\varphi(\ominus x \oplus z) = \varphi(\ominus x) \oplus \varphi(z) = \ominus w \oplus (w \oplus n) = n \in N$, which implies that $\ominus x \oplus z \in \varphi^{-1}(N)$. Put $y = \ominus x \oplus z$, then $z = x \oplus y = \varphi^{-1}(W) \oplus \varphi^{-1}(N)$.}
\end{proof}

\begin{thm}\cite[Theorem 10.1.16]{1}\label{t-88}
	Let $S = \{X_\alpha, p_\alpha^\beta: \alpha < \beta < \kappa\}$ be a continuous inverse spectrum, where every space $X_\alpha$ is compact and every bonding mapping $p_\alpha^{\alpha+1}$ is continuous, open, and has metrizable kernel.
	If $X_0$ is Dugundji, then so is the limit space $X$ of the spectrum $S$.
\end{thm}

An inverse spectrum, such as in Theorem~\ref{t-88} in which $X_0$ is a compact metrizable space, is called a {\color{blue} \it Haydon spectrum}. Now, we can prove one of main theorems in this section.

\begin{thm}\label{t-8}
	Let $H$ be a strong subgyrogroup of a strongly topological gyrogroup $G \in \mathscr{L}$. Then every compact $G_\delta$-set in the quotient space $G/H$ is Dugundji.
\end{thm}

\begin{proof}
	Let $X$ be a compact $G_\delta$-set in $G/H$.
	Then there exists a sequence $\{U_n: n \in \omega\}$ of open sets in $G/H$ such that $X = \bigcap_{n \in \omega}U_n$.
	Fix $n \in \omega$. We claim that there exist an element $L_n \in \mathscr{K}$ and an open set $V_n$ in $G/L_n$ such that
	\begin{equation}\label{eq:X}
		X \subseteq (\pi_{L_n}^H)^{-1}(V_n) \subseteq U_n.
		\tag{1}
	\end{equation}
	
Indeed, take any $x \in X$. Then $x \in U_n$ for each $n \in \omega$.
	Since the family $\{\pi_K^H: K \in \mathscr{K}\}$ generates the topology of the quotient space $G/H$ by Lemma~\ref{t-7}, there exist some $K_x \in \mathscr{K}$ and open set $W_x$ of $G/K_x$ such that $$x \in (\pi_{K_x}^H)^{-1}(W_x) \subseteq U_n.$$
	Put $O_x = (\pi_{K_x}^H)^{-1}(W_x)$.
	Each $O_x$ is open in $G/H$.
	Therefore, the family $\{O_x: x \in X\}$ forms an open covering of $X$.
	Since $X$ is compact, there exists a finite subset $C = \{x_1, x_2, \dots, x_m\} \subseteq X$ such that $X \subseteq \bigcup_{i=1}^m O_{x_i} \subseteq U_n$.
	Put $$L_n = \bigcap_{i=1}^m K_{x_i}.$$
	By Lemma \ref{t-6}, we have $L_n = \bigcap_{i=1}^m K_{x_i} \in \mathscr{K}$.
	Since $L_n \subseteq K_{x_i}$ for any $i = 1, 2, \dots, m$, it is natural to define mappings $p_i: G/L_n \rightarrow G/K_{x_i}$ satisfying $\pi_{K_{x_i}}^H = p_i \circ \pi_{L_n}^H$ for each $i= 1, 2, \dots, m$.
	Put $$V_n = \bigcup_{i=1}^m p_i^{-1}(W_{x_i}).$$
	Each $p_i^{-1}(W_{x_i})$ is open in $G/L_n$ since $p_i$ is continuous for each $i = 1, 2, \dots, m$.
	Thus $V_n$ is open in $G/{L_n}$.
	Finally, $$(\pi_{L_n}^H)^{-1}(V_n) = \bigcup_{i=1}^m (\pi_{L_n}^H)^{-1}\left( p_i^{-1}(W_{x_i}) \right) = \bigcup_{i=1}^m (\pi_{K_{x_i}}^H)^{-1}(W_{x_i}) = \bigcup_{i=1}^m O_{x_i}.$$
	Since $X \subseteq \bigcup_{i=1}^m O_{x_i} \subseteq U_n$, it follows that $X \subseteq (\pi_{L_n}^H)^{-1}(V_n) \subseteq U_n$. The proof of (1) is completed.
	
Let $L = \bigcap_{n \in \omega}L_n$.
Hence $L \in \mathscr{K}$ by Lemma~\ref{t-6}. We claim that
	$$X \subseteq (\pi_{L}^{H})^{-1}\pi_{L}^{H}(X)
	\subseteq \bigcap_{n=0}^{\infty} (\pi_{L_{n}}^{H})^{-1}\pi_{L_{n}}^{H}(X)
	\subseteq \bigcap_{n=0}^{\infty} (\pi_{L_{n}}^{H})^{-1}(V_{n})
	\subseteq \bigcap_{n=0}^{\infty} U_{n}
	= X.$$
	
Indeed, since $L \subseteq L_n$ for each $n \in \omega$, we can define a natural mapping $p_n: G/L \rightarrow G/L_n$ such that $\pi_{L_n}^H = p_n \circ \pi_L^H$.
	For any $z \in (\pi_L^H)^{-1}\pi_L^H(X)$, there exists some $x \in X$ such that  $\pi_L^H(z) = \pi_L^H(x)$.
	Then $p_n(\pi_L^H(z)) = p_n(\pi_L^H(x))$. Hence $\pi_{L_n}^H(z) = \pi_{L_n}^H(x)$, which implies  $z \in (\pi_{L_n}^H)^{-1}\pi_{L_n}^H(x) \subseteq (\pi_{L_n}^H)^{-1}\pi_{L_n}^H(X)$.
	Therefore, $(\pi_{L}^{H})^{-1}\pi_{L}^{H}(X) \subseteq (\pi_{L_{n}}^{H})^{-1}\pi_{L_{n}}^{H}(X)$ for each $n \in \omega$.
	Thus, $$(\pi_{L}^{H})^{-1}\pi_{L}^{H}(X) \subseteq \bigcap_{n=0}^{\infty} (\pi_{L_{n}}^{H})^{-1}\pi_{L_{n}}^{H}(X).$$
	Moreover, by (\ref{eq:X}), it is clear that $(\pi_{L_n}^H)^{-1}\pi_{L_n}^H(X) \subseteq (\pi_{L_n}^H)^{-1}(V_n) \subseteq U_n.$
	Hence $$\bigcap_{n=0}^{\infty} (\pi_{L_{n}}^{H})^{-1}\pi_{L_{n}}^{H}(X)
	\subseteq \bigcap_{n=0}^{\infty} (\pi_{L_{n}}^{H})^{-1}(V_{n})
	\subseteq \bigcap_{n=0}^{\infty} U_{n}
	= X.$$

Thus, $X = (\pi_{L}^{H})^{-1}\pi_{L}^{H}(X).$ Therefore, the restriction of $\pi_{L}^{H}$ to $X$ is also an open mapping of $X$ onto $\pi_{L}^{H}(X)$ since the mapping $\pi_{L}^{H}$ is open.
	
Now we will construct a Haydon spectrum $S$ with the limit space $X$.
	Consider the family $\mathscr{K}$ of $L$-subgyrogroups of $G$ defined before Lemma \ref{t-6}.
	We can enumerate the family $\mathscr{K}$ as $\mathscr{K} = \{K_\alpha: \alpha < \kappa\}$, where $\kappa$ is the cardinality of $\mathscr{K}$.
	Let $N_0 = L$, and $N_\alpha = L \cap \bigcap_{v < \alpha}K_v$, if $0 < \alpha < \kappa$.
	Given ordinals $\alpha$, $\beta$ with $\alpha < \beta < \kappa$, we shorten $\pi_{N_\alpha}^{N_\beta}$ to $\pi_\alpha^\beta$ and $\pi_{N_\alpha}^H$ to $\pi_\alpha$.
	Put also $X_\alpha = \pi_\alpha(X)$, and $p_\alpha^\beta = \pi_\alpha^\beta | X_\beta$.
	Let us show that $S = \{X_\alpha, p_\alpha^\beta: \alpha <\beta <\kappa\}$ is a Haydon spectrum with the limit space $X$. First, we prove the following claim.
	
	{\bf Claim 1.}
	The family $\{q_\alpha:\alpha < \kappa\}$ generates the topology of the compact space $X$, where $q_\alpha = \pi_\alpha|_{X}: X \rightarrow X_\alpha$ for each $\alpha < \kappa$. Moreover, $X$ is homeomorphic to the limit space of $S$.
	
Indeed, by Lemma~\ref{t-7}, the family $\{\pi_K^H: K \in \mathscr{K} \}$ generates the topology of $G/H$.
	Since $N_{\alpha + 1} = L \cap \bigcap_{v \leq \alpha}K_v \subseteq K_\alpha$, it is natural to define mappings $r_\alpha: G/N_{\alpha+1} \rightarrow G/K_\alpha$ satisfying $\pi_{K_\alpha}^H = r_\alpha \circ \pi_{\alpha+1}$ for any $\alpha < \kappa$.
	Let $W$ be an arbitrary open subset of $G/K_\alpha$.
	Then $$(\pi_{K_\alpha}^H)^{-1}(W) =  (\pi_{\alpha+1})^{-1}(r_\alpha^{-1}(W)),$$ where $r_\alpha^{-1}(W)$ is open in $G/N_{\alpha+1}$ since $r_\alpha$ is continuous.
	Thus, the topology of $G/H$ also can be generated by the family $\{\pi_\alpha: \alpha < \kappa \}$.
	By the definitions of $q_\alpha = \pi_\alpha|_{X}$ and $X_\alpha = \pi_\alpha(X)$, it follows that the family $\{q_\alpha: \alpha < \kappa \}$ generates the topology of $X$.
	
Next we prove that $X$ is homeomorphic to the limit space of $S$. In fact, the family $\{q_\alpha: \alpha < \kappa \}$ satisfies $q_\alpha = p_\alpha^\beta \circ q_\beta$ for each $\alpha < \beta$.
	Define a mapping $$e: X \rightarrow \prod_{\alpha < \kappa} X_\alpha, \qquad e(x) \mapsto (q_\alpha(x))_{\alpha < \kappa}, \mbox{where}\ x\in X.$$
	Clearly, $e$ is continuous.
	We first prove that $e$ is injective. Take any $x,y \in X$ and $x \neq y$; since X is Hausdorff, there exists an open set $U$ containing $x$ but not $y$.
	Since the topology of $X$ is generated by the family $\{q_\alpha: \alpha < \kappa\}$, there exist a finite set of indices $\{\alpha_1, \dots, \alpha_n\}$ and open sets $V_i \subseteq X_{\alpha_i}$ such that $x \in \bigcap_{i=1}^n q_{\alpha_i}^{-1}(V_i) \subseteq U$.
From $y \notin U$, it follows that $y \notin q_{\alpha_i}^{-1}(V_i)$ for some $i$, that is, $q_{\alpha_i}(y) \notin V_i$.
However, $x \in q_{\alpha_i}^{-1}(V_i)$, then $q_{\alpha_i}(x) \neq q_{\alpha_i}(y)$, which implies that $e(x) \neq e(y)$. Thus, $e$ is injective.
Since $q_\alpha = p_\alpha^\beta \circ q_\beta$ holds for all $\alpha < \beta < \kappa$, it follows that $e(X) \subseteq \varprojlim S$.
	Moreover, $e(X) = \varprojlim S$.
	To verify this, take any element $\xi = (\xi_\alpha)_{\alpha < \kappa} \in \varprojlim S$. By the definition of the inverse limit, $p_\alpha^\beta(\xi_\beta) = \xi_\alpha \in X_\alpha$ holds for each $\alpha < \beta < \kappa$.
	Let $K_\alpha = q_\alpha^{-1}(\{\xi_\alpha\})$ for each $\alpha < \kappa$.
	Then $\{K_\alpha\}_{\alpha < \kappa}$ forms a decreasing family of nonempty closed sets of $X$ (Indeed, for each $\alpha < \beta$, take any $x \in K_\beta$, we have $q_\beta(x) = \xi_\beta$. $q_\alpha(x) = p_\alpha^\beta(q_\beta(x)) = p_\alpha^\beta(\xi_\beta) = \xi_\alpha$. Then $x \in q_\alpha^{-1}(\xi_\alpha) = K_\alpha$. Thus, $K_\beta \subseteq K_\alpha$.).
By the compactness of $X$, we have $\bigcap_{\alpha < \kappa}K_\alpha \neq \emptyset$.
Pick any $x \in \bigcap_{\alpha < \kappa}K_\alpha$; then $x \in K_\alpha$ for each $\alpha < \kappa$,
	which implies that $q_\alpha(x) = \xi_\alpha$, that is, $e(x) = (q_\alpha(x))_{\alpha < \kappa} = (\xi_\alpha)_{\alpha < \kappa} = \xi$.
	Therefore, $e(X) = \varprojlim S$.
	Since $e$ is a continuous bijection from compact space $X$ to Hausdorff space $\varprojlim S$, we conclude that $e$ is a homeomorphism.
	Thus, $X$ is homeomorphic to the limit space of $S$.
	
	Now we verify that $S = \{X_\alpha, p_\alpha^\beta: \alpha <\beta <\kappa\}$ is a Haydon spectrum. Indeed, it suffices to show that $S$ satisfies the conditions of Theorem~\ref{t-88}. We divide the proof into the following Claims 2-6.

\smallskip
 {\bf Claim 2.} Each $X_\alpha$ is compact.

It is obvious  since $\pi_\alpha$ is continuous and $X$ is compact.

\smallskip	
{\bf Claim 3.} The spectrum $S$ is continuous.
	
Clearly, $p_\alpha^{\alpha+1}(X_{\alpha+1}) = X_\alpha$ for each $\alpha < \kappa$. Suppose that $\beta < \kappa$ is a limit ordinal,
it suffices to verify that the diagonal product $\underset{\alpha<\beta}{\Delta} p_{\alpha}^{\beta}$ of the family $\{p_\alpha^\beta:\alpha < \beta\}$ is a homeomorphism of $X_\beta$ onto the limit space $\varprojlim S_\beta$ of the spectrum $S_\beta$.
	In fact, $\underset{\alpha<\beta}{\Delta} p_{\alpha}^{\beta}$ is continuous.
Moreover, $N_\beta = \bigcap_{\alpha < \beta}N_\alpha$ for each $\beta < \kappa$;
Indeed, by the definition of $N_\alpha$, we have $$\bigcap_{\alpha < \beta}N_\alpha = \bigcap_{\alpha < \beta}(L \cap \bigcap_{v < \alpha}K_v) = L \cap (\bigcap_{\alpha < \beta}\bigcap_{v < \alpha}K_v).$$ Since $\beta$ is a limit ordinal, we have $\bigcap_{\alpha < \beta}\bigcap_{v < \alpha}K_v = \bigcap_{v < \beta}K_v$. Hence $$\bigcap_{\alpha < \beta}N_\alpha = L \cap (\bigcap_{\alpha < \beta}\bigcap_{v < \alpha}K_v) = L \cap \bigcap_{v < \beta}K_v = N_\beta.$$

We claim that $\underset{\alpha<\beta}{\Delta} p_{\alpha}^{\beta}$ is injective.
Indeed, suppose that $\underset{\alpha<\beta}{\Delta} p_{\alpha}^{\beta}(y) = \underset{\alpha<\beta}{\Delta} p_{\alpha}^{\beta}(y')$ for $y,y' \in X_\beta$. Then $p_{\alpha}^{\beta}(y) = p_{\alpha}^{\beta}(y')$ for all $\alpha < \beta$. Since $y, y' \in X_\beta$, there exist $x, x' \in X$ such that $y = \pi_\beta(x)$ and $y' = \pi_\beta(x')$. Hence $p_{\alpha}^{\beta}(y) = p_{\alpha}^{\beta}(\pi_\beta(x)) = \pi_\alpha(x)$ and $p_{\alpha}^{\beta}(y') = p_{\alpha}^{\beta}(\pi_\beta(x')) = \pi_\alpha(x')$. Then $\pi_\alpha(x) = \pi_\alpha(x')$, that is, $\ominus x' \oplus x \in N_\alpha$ for each $\alpha < \beta$, which implies that $\ominus x' \oplus x \in \bigcap_{\alpha < \beta}N_\alpha = N_\beta$. Therefore, $y = \pi_\beta(x) = x \oplus N_\beta = x' \oplus N_\beta = \pi_\beta(x') = y'$.

Thus, $\underset{\alpha<\beta}{\Delta} p_{\alpha}^{\beta}$ is a continuous bijection
from compact space $X_\beta$ to Hausdorff space $\varprojlim S_\beta$, hence it follows that $\underset{\alpha<\beta}{\Delta} p_{\alpha}^{\beta}$ is a homeomorphism.
Therefore, the spectrum $S$ is continuous.

\smallskip
{\bf Claim 4.} Each projection $p_\alpha^\beta$ of the spectrum $S$ is continuous and open.
	
First, we conclude that $X = (\pi_0)^{-1}\pi_0(X) = {\pi_0}^{-1}(X_0)$.
Indeed, it is clear that $X \subseteq (\pi_0)^{-1}\pi_0(X)$. Take any $y \in {\pi_0}^{-1}(X_0)$; there exists some $x \in X$ such that $\pi_0(y) = \pi_0(x)$, then $\pi_{L_n}^L(\pi_0(y)) = \pi_{L_n}^L(\pi_0(x))$ since $L \subseteq L_n$.
	Hence $\pi_{L_n}^H(y) = \pi_{L_n}^H(x)$ for each $n \in \omega$.
By (\ref{eq:X}), we have $\pi_{L_n}^H(y) = \pi_{L_n}^H(x) \in V_n$, that is, $y \in {\pi_{L_n}^H}^{-1}(V_n) \subseteq U_n$ for each $n \in \omega$. Thus, $y \in \bigcap_{n \in \omega}U_n = X$.
	
Then we can verify that $X = \pi_\alpha^{-1}(X_\alpha)$ for all $\alpha < \kappa$.
Indeed, since $N_\alpha \subseteq L$ for each $\alpha < \kappa$, we have $\pi_0 = \pi_L^{N_\alpha} \circ \pi_\alpha$.
		It is clear that $X \subseteq \pi_\alpha^{-1}(X_\alpha)$ holds for any $\alpha < \kappa$.
		Conversely, take any $y \in \pi_\alpha^{-1}(X_\alpha)$; then $\pi_\alpha(y) = \pi_\alpha(x)$ for some $x \in X$.
		Then $\pi_L^{N_\alpha}(\pi_\alpha(y)) =  \pi_L^{N_\alpha}(\pi_\alpha(x))$, that is, $\pi_0(y) = \pi_0(x) \in X_0$.
		Hence $y \in \pi_0^{-1}(X_0) = X$.
		Then $X = \pi_\alpha^{-1}(X_\alpha)$.

Consequently, we obtain $X_\beta = (\pi_\alpha^\beta)^{-1}(X_\alpha)$.
Indeed, for each $z \in X_\beta = \pi_\beta(X)$, there exists some $x \in X$ such that $z = \pi_\beta(x)$. Then $\pi_\alpha^\beta(z) = \pi_\alpha^\beta(\pi_\beta(x)) = \pi_\alpha(x) \in X_\alpha$, thus $z \in (\pi_\alpha^\beta)^{-1}(X_\alpha)$. Hence $X_\beta \subseteq (\pi_\alpha^\beta)^{-1}(X_\alpha)$. Conversely, for each $z \in (\pi_\alpha^\beta)^{-1}(X_\alpha)$, we have $\pi_\alpha^\beta(z) \in X_\alpha$,  then $z \in (\pi_\alpha^\beta)^{-1}(X_\alpha) \subseteq G/N_\beta$. Since $\pi_\beta$ is surjective, there exists some $y \in G$ such that $z = \pi_\beta(y)$. Then $\pi_\alpha(y) = \pi_\alpha^\beta\pi_\beta(y) = \pi_\alpha^\beta(z) \in X_\alpha$. Hence $y \in \pi_\alpha^{-1}(X_\alpha) = X$ and $z = \pi_\beta(y) \in \pi_\beta(X) = X_\beta$. Thus, $X_\beta = (\pi_\alpha^\beta)^{-1}(X_\alpha)$.

Now it remains to prove that each $p_\alpha^\beta$ is also open. Indeed, let $U$ be an open set in $X_\beta$. By the definition of subspace topology, there exists an open set $O  \subseteq G/N_\beta$ such that $U = O \cap X_\beta$. Consequently, we have $p_\alpha^\beta(U) = \pi_\alpha^\beta(O \cap X_\beta)$. It follows that $$\pi_\alpha^\beta(O \cap X_\beta) = \pi_\alpha^\beta(O \cap (\pi_\alpha^\beta)^{-1}(X_\alpha)) = \pi_\alpha^\beta(O) \cap X_\alpha.$$ Since the mapping $\pi_\alpha^\beta$ is open, we conclude that $\pi_\alpha^\beta(O)$ is open in $G/N_\alpha$, which implies that $p_\alpha^\beta(U)$ is open in $X_\alpha$. Therefore, $p_\alpha^\beta$ is open.
	
\smallskip
{\bf Claim 5.} Each bonding mapping $p_\alpha^{\alpha+1}$ has metrizable kernel.
	
	It suffices to construct a metrizable compact space $C_\alpha$ and a topological embedding $i_\alpha: X_{\alpha+1} \rightarrow X_\alpha \times C_\alpha$ such that $p_\alpha^{\alpha+1} = p \circ i_\alpha$, where $p: X_\alpha \times C_\alpha \rightarrow X_\alpha$ is the projection.
	By definition, $N_\alpha = L \cap \bigcap_{\gamma < \alpha}K_\gamma$, which implies  $N_{\alpha+1} = N_\alpha \cap K_\alpha$.
	We define the natural projections $\pi_\alpha^{\alpha+1}:G/N_{\alpha+1} \rightarrow G/N_\alpha$ and $\pi_{K_\alpha}^{N_{\alpha+1}}: G/N_{\alpha+1} \rightarrow G/K_\alpha$.
	It can be verified that the diagonal product
	$$\pi_\alpha^{\alpha+1} \Delta \pi_{K_\alpha}^{N_{\alpha+1}}: G/N_{\alpha+1} \to G/N_\alpha \times G/K_\alpha, \quad g \oplus N_{\alpha+1} \mapsto (g \oplus N_\alpha, g \oplus K_\alpha)$$ is injective.
Indeed, suppose that $\pi_\alpha^{\alpha+1} \Delta \pi_{K_\alpha}^{N_{\alpha+1}}(g \oplus N_{\alpha+1}) = \pi_\alpha^{\alpha+1} \Delta \pi_{K_\alpha}^{N_{\alpha+1}}(h \oplus N_{\alpha+1})$, we have $g \oplus N_\alpha = h \oplus N_\alpha$ and $g \oplus K_\alpha = h \oplus K_\alpha$, which implies that $\ominus h \oplus g \in N_\alpha$ and $\ominus h \oplus g \in K_\alpha$. Then $\ominus h \oplus g \in N_\alpha \cap K_\alpha = N_{\alpha+1}$, that is, $g \oplus N_{\alpha+1} = h \oplus N_{\alpha+1}$. Thus, $\pi_\alpha^{\alpha+1} \Delta \pi_{K_\alpha}^{N_{\alpha+1}}$ is injective.

Let $f_\alpha = \pi_{K_\alpha}^{N_{\alpha+1}} |_{X_{\alpha+1}}: X_{\alpha+1} \rightarrow G/K_\alpha$ be the restriction of $\pi_{K_\alpha}^{N_{\alpha+1}}$ to $X_{\alpha+1}$.
	Put $$C_\alpha = f_\alpha(X_{\alpha+1}) \subseteq G/K_\alpha.$$
	Clearly, $C_\alpha$ is compact.
By the definition of $\mathscr{K}$, $G/K_\alpha$ is submetrizable. Since any subspace of a submetrizable space is submetrizable, it follows that $C_\alpha$ as the compact subspace of $G/K_\alpha$ is metrizable.
	
	Define the diagonal product $$i_\alpha = p_\alpha^{\alpha+1} \Delta f_\alpha: X_{\alpha+1} \rightarrow X_\alpha \times C_\alpha,\quad i_\alpha(y) = (p_\alpha^{\alpha+1}(y), f_\alpha(y)).$$
	It is obvious that $i_\alpha$ is continuous.
	Since $\pi_\alpha^{\alpha+1} \Delta \pi_{K_\alpha}^{N_{\alpha+1}}$ is injective, we conclude that
	$i_\alpha$ as the restriction of $\pi_\alpha^{\alpha+1} \Delta \pi_{K_\alpha}^{N_{\alpha+1}}$ to $X_{\alpha+1}$ is also injective.
	Hence $i_\alpha$ is a continuous bijection from compact space $X_{\alpha+1}$ to Hausdorff space $X_\alpha \times C_\alpha$.
	Therefore, $i_\alpha$ is a topological embedding.
Moreover, it is obvious that $p_\alpha^{\alpha+1} = p \circ i_\alpha$ holds. Thus, each bonding mapping $p_\alpha^{\alpha+1}$ has metrizable kernel.

\smallskip	
{\bf Claim 6.} The space $X_0$ is compact metrizable.
	
	Note that $X_0 = \pi_0(X) = \pi_L(X)$.
	Since $L \in \mathscr{K}$, the quotient space $G/L$ is submetrizable.
	As a compact subspace of the submetrizable space $G/L$, we see that $X_0$ is metrizable.
	Consequently, $X_0$ is Dugundji by Propositon~\ref{t-7777}.
	
	It follows from Theorem~\ref{t-88} that $S = \{X_\alpha, p_\alpha^\beta: \alpha <\beta <\kappa\}$ is a Haydon spectrum with the limit space $X$.
	Moreover, since $X_0$ is Dugundji, the limit space $X$ of the spectrum $S$ is also Dugundji. Therefore, every compact $G_\delta$-set in the quotient space $G/H$ is Dugundji.
\end{proof}

Next, we prove the second main theorem in this section.

\begin{thm}\label{t-9}
	Let $G \in \mathscr{L}$ be a strongly topological gyrogroup and $N$ be a closed strong subgyrogroup of $G$. If the quotient space $G/N$ is locally compact, then the inequality $w(G/N) \leq c$ is equivalent to the separability of $G/N$.
\end{thm}

\begin{proof}
	Let $\pi: G \rightarrow G/N$ be the quotient mapping of $G$ onto the locally compact left coset space $G/N$. From \cite[Theorem 3.7]{10}, $G/N$ is regular. If $G/N$ is separable, then $w(G/N) \leq c$. Assume therefore that $w(G/N) \leq c$.
	
	Since $G/N$ is locally compact, there exists an open neighborhood $V$ of $\pi(0)$ in $G/N$ such that $\overline{V}$ is compact.
	Take an open neighborhood $U$ of $0$ in $G$ such that $$\pi(0) \in \pi(U) \subseteq V.$$
	Then $\overline{\pi(U)} \subseteq \overline{V}$,
	and hence $\overline{\pi(U)}$ is compact.
	Since $G$ is Lindel\"of, it is $\omega$-narrow, there exists a countable set $C \subseteq G$ such that $G = C \oplus U$.
	Then $$G/N = \pi(G) = \pi(C \oplus U) = \bigcup_{x \in C}\pi(x \oplus U) = \bigcup_{x \in C}\overline{\pi(x \oplus U)}.$$
	
	Define a mapping $h_x: G/N \rightarrow G/N$ by $h_x(y \oplus N) = (x \oplus y) \oplus N$ for each $x, y \in G$.
	Then $h_x(\pi(U)) = h_x(U \oplus N) = (x \oplus U) \oplus N = \pi(x \oplus U).$
	By \cite[Theorem 3.13]{9}, $h_x$ is homeomorphism, it follows that $$h_x(\overline{\pi(U)}) = \overline{h_x(\pi(U))} = \overline{\pi(x \oplus U)}.$$
	Hence the compact sets $\overline{\pi(x \oplus U)}$, with $x \in C$, cover the space $G/N$.
	Since $C$ is a countable set, it follows that $G/N$ is $\sigma$-compact.
	
	Let $\{K_n: n \in \omega \}$ be a countable family of compact sets that covers $G/N$, that is, $G/N = \bigcup_{n \in \omega}K_n$.
	Fix any $n \in\omega$. For any $x \in K_n$, by the local compactness of $G/N$ we can find an open neighborhood $U_x$ of $x$ in $G/N$ such that $\overline{U_x}$ is compact.
	Then $K_n \subseteq \bigcup_{x \in K_n}U_x$,
	the sequence $\{U_x: x \in K_n \}$ forms an open cover of the compact set $K_n$. Therefore, there exists a finite set $D = \{x_1, x_2, \cdots, x_m\}$ such that $K_n \subseteq \bigcup_{i \leq m}U_{x_i}$.
	Put $O_n = \bigcup_{i \leq m}U_{x_i}$.
	Then $O_n$ is open in $G/N$ and $\overline{O_n} = \bigcup_{i \leq m}\overline{U_{x_i}}$ is also compact, as the union of finitely many compact sets.
	Therefore, making use of the local compactness of $G/N$ we can find, for every $n \in \omega$, an open set $O_n$ with compact closure in $G/N$ such that $K_n \subseteq O_n$.
	
	Since the space $G/N$ is normal, and $K_n\subseteq O_n$, with $K_n$ closed and
	$O_n$ open, we can choose an open set $V_0$ such that $K_n\subseteq V_0\subseteq \overline{V_0}\subseteq O_n$.
	Inductively, choose open sets $V_{n+1}$ such that $K_n \subseteq V_{n+1}\subseteq \overline{V_{n+1}}\subseteq V_n$, where $n \in \omega.$
	Put $$F_n = \bigcap_{n \in \omega}V_n.$$
	Thus, $F_n$ is a $G_\delta$-set in $G/N$ such that $K_n \subseteq F_n \subseteq O_n$.
	Moreover, by the choice of $V_n$, $\bigcap_{n \in \omega}V_n = \bigcap_{n \in \omega}\overline{V_n}.$
	So $F_n$ is closed. Hence $F_n$ is a closed $G_\delta$-set satisfying $K_n \subseteq F_n \subseteq O_n$.
	And $\overline{O_n}$ is compact, which implies that $F_n$ is compact for each $n \in \omega$. %{\color{blue} $F_n$ is closed, and $F_n \subseteq O_n \subseteq \overline{O_n}$, then $F_n$ is closed in $\overline{O_n}$. $\overline{O_n}$ is compact, therefore $F_n$ is compact.}
	
	Summing up, each $F_n$ is a compact $G_\delta$-set in the quotient space $G/N$ of  strongly topological gyrogroup $G$.
	By Theorem~\ref{t-777} and Theorem~\ref{t-8}, each $F_n$ is a dyadic compact space.
	As $w(F_n) \leq w(G/N) \leq c$ for each $n \in \omega$, it follows from \cite[Proposition 2.2]{11} that each $F_n$ is separable.
	The inclusions $K_n \subseteq F_n$ with $n \in \omega$ imply that $G/N = \bigcup_{n \in \omega}F_n$, so the space $G/N$ is separable.
\end{proof}


\begin{thebibliography}{99}\setlength{\itemsep}{1.2ex}
\bibitem{1} A.V. Arhangel'ski\v\i, M. Tkachenko, {\it Topological Groups and Related Structures}, Atlantis Press and World Sci., 2008.

\bibitem{2} W. Atiponrat, Topological gyrogroups: generalization of topological groups, Topol. Appl., 224 (2017) 73-82.

\bibitem{5} M. Bao, F. Lin, Submetrizability of strongly topological gyrogroups, Houst. J. Math. (2020)6132.

\bibitem{6} M. Bao, F. Lin, Feathered gyrogroups and gyrogroups with countable pseudocharacter, Filomat, 33(16)(2019) 5113-5124.

\bibitem{9} M. Bao, X. Xu, A note on (strongly) topological gyrogroups, Topol. Appl., 307 (2022) 107950.

\bibitem{7} R. Engeling, {\it General Topology}, Heldermann Verlag, Berlin, 1989.

\bibitem{10} Y. Jin, L. Xie, Quotients with respect to strongly L-subgyrogroups, arXiv preprint arXiv:2210.03648 (2022)

\bibitem{14} S. Lai, F. Lin, The separability embedding of $\sigma$-compact strongly topological gyrogroups, arXiv preprint arXiv:2605.23441 (2026).

\bibitem{11} A. Leiderman, S. Morris, M. Tkachenko, Density character of subgroups of topological groups, Trans. Amer. Math. Soc., 369 (2017) 5645-5664.

\bibitem{3} E. Reznichenko, Almost paratopological groups, Topol. Appl., 338 (2023) 108673.

\bibitem{12} T. Suksumran, K. Wiboonton, Isomorphism theorems for gyrogroups and L-subgyrogroups, J. Geom. Symmetry Phys., 37, (2015) 67-83.

\bibitem{13} T. Suksumran, {\it The algebra of gyrogroups: Cayley's theorem, Lagrange's theorem, and isomorphism theorems}, Essays Math. Appl., (2016) 369-437.

\bibitem{4} M. Tkachenko, Lindel\"of $\Sigma$-Spaces and R-Factorizable Paratopological Groups, Axioms 4(3)(2015) 254-267.

\bibitem{16} A. A. Ungar, {\it Analytic hyperbolic geometry and Albert Einstein's special theory of relativity}, World scientific, 2008.

\bibitem{15} J. Wattanapan, W. Atiponrat, T. Suksumran, Embedding of locally compact Hausdorff topological gyrogroups in topological groups, Topol. Appl., 273 (2020) 107102.



\end{thebibliography}
\end{document}